\documentclass[11pt,reqno]{amsart}

\usepackage[a4paper,top=3cm,bottom=3cm,left=2.7cm,right=2.7cm]{geometry}

\usepackage{amssymb}
\usepackage{amsmath}
\usepackage{mathtools}
\usepackage{mathrsfs}
\usepackage{enumerate}

\usepackage[colorlinks=true,
            citecolor=red,
            linkcolor=blue,
            urlcolor=red,
            bookmarksnumbered=true,
            bookmarksopen=true]{hyperref}

\numberwithin{equation}{section}

\newcommand{\Ff}{\mathcal{F}}
\newcommand{\Gg}{\mathcal{G}}
\newcommand{\Ee}{\mathcal{E}}
\newcommand{\Oo}{\mathcal{O}}
\newcommand{\Hom}{\operatorname{Hom}}
\newcommand{\rk}{\operatorname{rank}}
\newcommand{\Spec}{\operatorname{Spec}}

\newcommand{\Mor}{\operatorname{Mor}}

\newtheorem{thm}{Theorem}[section]

\newtheorem{lem}[thm]{Lemma}

\theoremstyle{definition}
\newtheorem{defn}[thm]{Definition}
\newtheorem{rem}[thm]{Remark}

\newtheorem{nota}[thm]{Notation}
\newtheorem{setup}[thm]{Set-up}

\begin{document}

\title[Optimal bend-and-break for foliations]{Optimal bounds in bend-and-break for foliations}

\author{Jihao Liu, Zeming Sun, Jiedong Jiang}
\address{School of Mathematical Sciences, Peking University, No.~5 Yiheyuan Road, Haidian District, Beijing 100871, China}
\address{Beijing International Center for Mathematical Research, Peking University, No.~5 Yiheyuan Road, Haidian District, Beijing 100871, China}
\email{liujihao@math.pku.edu.cn}
\email{zeming@kurims.kyoto-u.ac.jp}
\address{Westlake Institute for Advanced Study, Westlake University, No. 600 Dunyu Road, Sandun town, Xihu district, Hangzhou, Zhejiang, 310030}
\email{jiangjiedong@westlake.edu.cn}

\date{\today}

\begin{abstract}
We show that for every foliation \(\Ff\) of rank \(r\) on a normal projective variety, the optimal constant in the bend-and-break inequality for tangent rational curves is \(r+1\). The proof combines the method of Bogomolov--McQuillan and the bend-and-shatter method developed by Jovinelly--Lehmann--Riedl. The proof of the main result of this paper substantially uses generative AI, particularly the Rethlas system.
\end{abstract}

\maketitle

\tableofcontents

\section{Introduction}\label{sec:intro}

We work over the field of complex numbers $\mathbb C$, unless otherwise specified.
The proof of our main theorem nevertheless passes through positive characteristic via Frobenius pullback and spreading-out arguments; for the relevant intermediate results (in Sections~\ref{sec:stable-maps} and~\ref{sec:char-p-BB}) we are explicit about the working field whenever this matters.

\subsection{Bend-and-break bounds}

Mori's bend-and-break method~\cite{Mor82} is one of the cornerstones of modern birational geometry: given a smooth projective variety \(X\) of dimension \(n\) and a curve \(C\subset X\) of \(K_X\)-negative degree, it produces a rational curve through any chosen point of \(C\) of explicitly bounded degree.
In its sharpest form, for every nef divisor \(H\) on \(X\), through every point of \(C\) there exists a rational curve \(\Sigma\subset X\) with
\[
H\cdot\Sigma \;\le\; (n+1)\,\frac{H\cdot C}{-K_X\cdot C},
\]
and the constant \(n+1\) is optimal.
The weaker bound with constant \(2n\) is due to~\cite[Theorem~5]{MM86} (see also~\cite[Theorem~II.5.8]{Kol96} for the extension to nef \(\mathbb R\)-Cartier \(H\)).
The optimal constant \(n+1\) was obtained for ample \(H\) by~\cite{Mor82} (see~\cite[Theorem~1.13]{KM98} for a textbook treatment).
Combined with the cone theorem, this bend-and-break inequality bounds the length of \(K_X\)-negative extremal rays by \(n+1\), and underlies the construction of contraction morphisms, Fano fibrations, and the inductive proofs of the minimal model program; see~\cite{Kol96,KM98} for textbook treatments.

For singular varieties, the optimal bend-and-break constant remained out of reach for a long time.
Kawamata~\cite{Kaw91} obtained the analogous length bound for klt pairs with the weaker constant \(2n\), and a sequence of subsequent works extended and refined this bound for log canonical pairs and more general singular settings; see~\cite{Fuj11,Fuj17} and the references therein.
The optimal version was settled very recently, in the strongest possible form, by Jovinelly--Lehmann--Riedl~\cite{JLR25b}: for every log canonical pair \((X,\Delta)\) of dimension \(n\) and every sufficiently general \((K_X+\Delta)\)-negative complete-intersection curve \(C\), through a general point of \(C\) there exists a rational curve \(\Sigma\) with
\[
H\cdot\Sigma \;\le\; (n+1)\,\frac{H\cdot C}{-(K_X+\Delta)\cdot C}.
\]
Their breakthrough also supplies the central technical input, a relative bend-and-shatter lemma, used in the present paper.
The bound of~\cite{JLR25b} has already led to several important applications in birational geometry, such as the construction of free curves in the smooth locus of terminal Fano threefolds \cite{JLR25a} and the characterization of $\mathbb P^n$ \cite{FJLR26}.

\subsection{Bend and break for foliations}
For \emph{foliations}, the analogue of bend-and-break reads as follows.
Let \(\Ff\subset T_X\) be a foliation of rank \(r\) on a normal projective variety \(X\) with canonical divisor $K_{\Ff}$.
The natural counterpart of Mori's question asks for the optimal constant in the bend-and-break inequality for rational curves \emph{tangent} to \(\Ff\).
The history of this question starts with Miyaoka~\cite{Miy87}, who introduced the technique of deforming a curve along a foliation in characteristic \(p\); this was refined by Shepherd-Barron~\cite{SB92} into a foliated bend-and-break inequality with the constant \(2n\), and improved by Bogomolov--McQuillan~\cite{BM16} to the better constant \(2r\).
Until now, the optimal bound \(r+1\) has been out of reach in full generality.

The main result of this paper resolves the optimal bend-and-break question for arbitrary foliations.

\begin{thm}[Main theorem]\label{thm:main}
Let \(X\) be a normal projective variety of dimension \(n\), and let \(\Ff\) be a foliation on \(X\) of rank \(r\).
Let \(H_1,\ldots,H_{n-1},H\) be ample divisors on \(X\), and let \(C\) be a general intersection of elements \(D_i\in|m_iH_i|\) for \(1\le i\le n-1\), with \(m_i\gg 0\).
Suppose that
\[
K_{\Ff}\cdot C \;<\; 0.
\]
Then through a general point of \(C\) there exists a rational curve \(\Sigma\) tangent to \(\Ff\) with
\[
H\cdot\Sigma \;\le\; (r+1)\,\frac{H\cdot C}{-K_{\Ff}\cdot C}.
\]
\end{thm}

Theorem~\ref{thm:main} is the optimal bend-and-break inequality for foliations.
It sharpens the Bogomolov--McQuillan constant \(2r\) of~\cite{BM16} to the optimal \(r+1\), and the two extreme cases recover known results: the case \(r=n\) (\(\Ff=T_X\)) corresponds to the bend-and-break bound for varieties \cite{JLR25b} and the case \(r=1\) recovers the Bogomolov--McQuillan bound \cite{BM16} (as $2r=r+1=2$).
Whenever a cone theorem for \(\Ff\) becomes available, Theorem~\ref{thm:main} translates into the optimal length bound \(r+1\) for \(K_{\Ff}\)-negative extremal rays.
However, unlike in the smooth and log canonical settings, the cone theorem for foliations (the framework that converts bend-and-break into a length bound for extremal rays) has been proved only in special cases (see for instance~\cite{ACSS21,CHLX23,Spi20,CS21,CS25a,CS25b}), and is not available in the generality treated here.
We therefore state our main result as a bend-and-break inequality, and the corresponding length statement for \(K_{\Ff}\)-negative extremal rays follows in any setting in which a foliated cone theorem is available.

The constant \(r+1\) in Theorem~\ref{thm:main} is itself optimal, in the following sense.

\begin{rem}[Optimality of the constant \(r+1\)]\label{rem:optimality}
The constant \(r+1\) in Theorem~\ref{thm:main} is optimal: no strictly smaller constant suffices, even after restricting to smooth, algebraically integrable foliations.

To see this, fix integers \(1\le r\le n-1\) and let \(X=\mathbb{P}^r\times\mathbb{P}^{n-r}\), with \(\xi_1,\xi_2\) the hyperplane classes pulled back from the two factors.
Let \(\Ff = T_{X/\mathbb{P}^{n-r}}\) be the rank \(r\) foliation induced by the smooth fibration \(X\to\mathbb{P}^{n-r}\); its leaves are the \(\mathbb{P}^r\)-fibers of the second projection, and \(\Ff\) is regular everywhere.
A direct computation gives
\[
K_{\Ff} \;=\; -(r+1)\,\xi_1.
\]

For any \(\varepsilon\in(0,1)\cap\mathbb{Q}\), set \(H_1=\cdots=H_{n-1}=\xi_1+\xi_2\) and \(H_\varepsilon = \xi_1+\varepsilon\,\xi_2\); both are ample.
A general complete-intersection curve \(C\) of large multiples of the \(H_i\) has class proportional to \((\xi_1+\xi_2)^{n-1}\), and a direct computation on \(\mathbb{P}^r\times\mathbb{P}^{n-r}\) gives
\[
\frac{H_\varepsilon\cdot C}{-K_{\Ff}\cdot C}
\;=\;
\frac{1}{r+1}\left(1+\frac{\varepsilon(n-r)}{r}\right).
\]
The shortest \(\Ff\)-tangent rational curves through a general point of \(C\) are the lines contained in the \(\mathbb{P}^r\)-fibers of the second projection; any such line \(\ell\) satisfies \(\xi_2\cdot\ell=0\), so \(H_\varepsilon\cdot\ell=\xi_1\cdot\ell=1\).
Suppose Theorem~\ref{thm:main} held with some constant \(c<r+1\) in place of \(r+1\).
Then for sufficiently small \(\varepsilon>0\) the bound would force
\[
H_\varepsilon\cdot\Sigma \;\le\; c\,\frac{H_\varepsilon\cdot C}{-K_{\Ff}\cdot C} \;<\; 1,
\]
which is impossible, since any nonconstant \(\Ff\)-tangent curve has \(H_\varepsilon\)-degree at least \(1\).
Hence the constant \(r+1\) cannot be improved.
\end{rem}

\subsection{Strategy of the proof}\label{subsec:strategy}

The proof combines three ingredients.
\begin{enumerate}[(i)]
\item \emph{The optimal bend-and-shatter lemma~\cite[Lemma~2.1]{JLR25b}}: a \(k\)-dimensional family of stable maps from a fixed marked source, with the first \(k\) marked images held fixed, must degenerate to a stable-map limit that contains \(k\) non-contracted rational components through the marked points.
This improves Mori's classical bend-and-break (in the form of~\cite[Theorem~II.5.8]{Kol96}) by passing to the \(k\)-pointed Kontsevich stable-map stack.

\item \emph{The graph-neighborhood relative model~\cite{BM16,KST07}}: given a rank \(s\) foliation \(\Gg\subset T_X\) regular along a smooth complete curve \(C\subset X_{\rm reg}\) with \(\Gg|_C\) ample, the BM/KST construction algebraizes the graph-neighborhood of the \(\Gg\)-leaves along \(C\) and produces, after passing to a finite cover \(\nu\colon B\to C\) of positive genus, a smooth projective relative model
\[
\rho\colon W\to B,\qquad \phi\colon W\to X,\qquad \sigma\colon B\to W,
\]
where \(\rho\) is flat of relative dimension \(s\) and smooth near a section \(\sigma(B)\), the relative tangent bundle along the section is the pulled-back foliation, and the general fiber components near \(\sigma(B)\) are the (rationally connected, \(s\)-dimensional) leaves of \(\Gg\).
In particular foliation tangency in \(X\) corresponds to \(\rho\)-verticality in \(W\).

\item \emph{The Campana--P\u{a}un algebraicity criterion~\cite{CP19}}: combining the MR restriction theorem with the Campana--P\u{a}un slope-positive algebraicity theorem, this produces, from a \(K_{\Ff}\)-negative sufficiently general complete-intersection curve \(C\) on \(X\), a positive-rank algebraically integrable subfoliation \(\Gg\subset\Ff\) of rank \(s\le r\) such that \(\Gg|_C\) is ample and \(-K_{\Gg}\cdot C\ge -K_{\Ff}\cdot C\).
\end{enumerate}

The novelty of this paper is the observation that these three ingredients compose into the optimal bound.
The bend-and-shatter lemma in~(i) is absolute, while the model in~(ii) is relative over the base curve \(B\); the bridge between them is a relative form of bend-and-shatter (Lemma~\ref{lem:relative-jlr-proposition-22}), applied to the relative Hom scheme \(\Hom_B(C,W;h)\) of \(B\)-morphisms with prescribed composition to \(B\).
A Frobenius pullback amplifies the dimension of this relative Hom scheme so that the relative bend-and-shatter lemma applies (Lemma~\ref{lem:bm-jlr-relative-positive-characteristic}), and a closed-locus form of the spreading-out argument (Lemma~\ref{lem:characteristic-zero-vertical-spreading}) preserves verticality during the passage from positive characteristic back to characteristic zero.
Pushing forward by \(\phi\) converts the resulting \(\rho\)-vertical rational curves into \(\Gg\)-tangent rational curves through a general point of \(C\); the rank slack \(s\le r\) and the degree slack \(-K_{\Gg}\cdot C\ge -K_{\Ff}\cdot C\) from~(iii) then upgrade the bound \((s+1)\,H\cdot C/(-K_{\Gg}\cdot C)\) for \(\Gg\)-tangent curves into the optimal bound \((r+1)\,H\cdot C/(-K_{\Ff}\cdot C)\) for \(\Ff\)-tangent curves.

The proof is organized in four steps.

\smallskip\noindent\textit{Step 1: Reduction to the algebraically integrable case.}
Since \(K_{\Ff}\cdot C<0\), the restriction \(\Ff|_C\) has positive degree.
The Mehta--Ramanathan (MR) restriction theorem for the Harder--Narasimhan (HN) filtration~\cite[Notation~4]{KST07} produces a saturated subsheaf \(\Gg\subset\Ff\) whose restriction to \(C\) is the positive part of the HN filtration of \(\Ff|_C\).
The Campana--P\u{a}un algebraicity criterion~\cite[Lemma~4.12, Theorem~4.2]{CP19} then promotes \(\Gg\) to a positive-rank algebraically integrable subfoliation \(\Gg\subset\Ff\), with \(\Gg|_C\) ample and \(-K_{\Gg}\cdot C\ge -K_{\Ff}\cdot C\).

\smallskip\noindent\textit{Step 2: The graph-neighborhood relative model.}
Following the BM/KST graph-neighborhood construction~\cite{BM16,KST07}, after a finite cover \(\nu\colon B\to C\) from a smooth projective curve of positive genus and the product trick on \(X\times B\), we algebraize the analytic graph-neighborhood of the section.
The result, after normalization and projective desingularization, is a smooth projective relative model
\[
\rho\colon W\to B,\qquad \phi\colon W\to X,\qquad \sigma\colon B\to W,
\]
with \(\rho\) projective and flat of relative dimension \(s=\rk\Gg\), smooth near \(\sigma(B)\), satisfying \(\phi\circ\sigma=\nu\) and \(T_{W/B}|_{\sigma(B)}\simeq(\phi\circ\sigma)^*\Gg\); for general \(b\in B\), the fiber component through \(\sigma(b)\) maps into the \(\Gg\)-leaf through \(\phi(\sigma(b))\).

\smallskip\noindent\textit{Step 3: Characteristic \(p\) bend-and-shatter.}
We spread the relative model over a finitely generated \(\mathbb Z\)-subalgebra of \(\mathbb C\) and reduce modulo a sufficiently large prime \(p\).
Set \(\delta := \deg_B\sigma^*T_{W/B} = e(-K_{\Gg}\cdot C)>0\).
Form the geometric Frobenius twist \(B_m=B^{(p^{-m})}\) and the composite \(\sigma_m=\sigma\circ\Phi_m\colon B_m\to W\), where \(\Phi_m\colon B_m\to B\) is the geometric Frobenius of degree \(q_m=p^m\).
Let \(M_m\) be an irreducible component of the relative Hom scheme \(\Hom_B(B_m,W;\rho\circ\sigma_m)\) through \(\sigma_m\).
The relative form of the Bogomolov--McQuillan tangent-obstruction calculation~\cite[Proposition~3.1.1]{BM16} yields
\[
\dim M_m \;\ge\; \chi(B_m,\sigma_m^*T_{W/B}) \;=\; q_m\delta+s(1-g(B)).
\]
The Jovinelly--Lehmann--Riedl bend-and-shatter lemma~\cite[Lemma~2.1]{JLR25b} then produces, for general \(k_m\)-tuples of marked points with \(k_m=\lfloor\dim M_m/(s+1)\rfloor\), a stable-map limit containing \(k_m\) non-contracted vertical rational curves through the marked images.
Pigeonholing degrees produces a vertical rational curve \(R_m\) of small \(L\)-degree, where \(L\) is an ample Cartier divisor; as \(m\to\infty\),
\[
\frac{q_m}{k_m} \;\to\; \frac{s+1}{\delta},\qquad L\cdot R_m \;\to\; (s+1)\,\frac{L\cdot \sigma_*B}{\delta}.
\]

\smallskip\noindent\textit{Step 4: Characteristic zero spreading.}
The closed-locus form of the spreading argument in the proof of~\cite[Theorem~1.1]{JLR25b}, which builds on~\cite[Theorem~5]{MM86} and~\cite[Theorem~II.5.8]{Kol96}, lifts the bounded-degree vertical rational-curve locus from a Zariski-dense set of positive-characteristic fibers back to the characteristic-zero generic fiber, and hence to \(\mathbb C\).
This yields a rational curve \(\Sigma\) tangent to \(\Gg\) through a general point of \(C\), with the prescribed \(H\)-degree bound.
Since \(\Gg\subset\Ff\) and \(s\le r\), the curve \(\Sigma\) is also tangent to \(\Ff\), and the bound passes from \((s+1)/(-K_{\Gg}\cdot C)\) to \((r+1)/(-K_{\Ff}\cdot C)\).

\subsection{Outline of the paper}

Section~\ref{sec:notation} sets up notation and recalls the basics of foliations and the Mehta--Ramanathan restriction theorem.
Section~\ref{sec:positive-HN} carries out the positive Harder--Narasimhan reduction.
Section~\ref{sec:graph-neighborhood} builds the relative graph-neighborhood model.
Section~\ref{sec:stable-maps} establishes the stable-map preliminaries.
Section~\ref{sec:char-p-BB} carries out the relative bend-and-shatter and the characteristic-\(p\) Bogomolov--McQuillan/Jovinelly--Lehmann--Riedl bound.
Section~\ref{sec:spreading} proves the characteristic-zero spreading lemma and deduces the optimal bend-and-break bound for ample subfoliations.
Section~\ref{sec:proof} completes the proof of Theorem~\ref{thm:main}.

\subsection{Use of generative AI}

The question in this paper, the optimal bend-and-break bound for foliations, was discussed by the first author, Paolo Cascini, and Calum Spicer when the first author was in London in January 2026.
The reason is simple: with our understanding of \cite{BM16} and since \cite{JLR25b} is very short and easy to read, it seems that ``something could work by just putting them together''.
However, \cite{BM16,JLR25b} are written in completely different styles (particularly considering that the first version of \cite{BM16} was posted in 2001).
Putting them together requires a significant amount of work and energy.

In March 2026, the first author of the paper discussed with Florin Ambro, Federico Bongiorno, and Paolo Cascini what birational geometry questions generative AI can address.
The optimal bend-and-break question was raised again at that time.
Later, the Rethlas system was introduced to the first author.
The prompt was simple: \emph{can you combine the methodologies of Bogomolov--McQuillan and Jovinelly--Lehmann--Riedl to obtain an optimal bend-and-break bound?}
After running the system for 12 hours, Rethlas obtained a complete proof.

Like other AI-generated papers, the writing of the AI is terrible.
However, it is very easy to check that the key bridge that connects the key ideas of \cite{BM16} and \cite{JLR25b}, Lemma~\ref{lem:relative-jlr-proposition-22}, is correct (and we are very grateful to Eric Riedl, who was visiting Peking University in April, for useful related comments).
With this, we are confident in the correctness of the paper despite the awkward writing at that time.
In particular, as Eric Riedl and Calum Spicer pointed out to us, the proof of Section~\ref{sec:positive-HN} was overcomplicated, as Rethlas did not find \cite{CP19}, and was later rewritten and simplified.
The verification of all details of the paper was done by all three authors. 

See \cite{Ju+26} for a detailed introduction to the Rethlas system. Due to the limitation of generative AI, it is possible that we have missed some related references in the literature, and we welcome any comments from experts.

\subsection*{Acknowledgements}
This work is supported by the National Key R\&D Program of China \#\allowbreak 2024YFA1014400.
The authors would like to thank the Rethlas team, namely Haocheng Ju, Shurui Liu, Guoxiong Gao, Yuefeng Wang, Bin Wu, Liang Xiao, and Bin Dong, for their contributions to the development of Rethlas and its customized version used for the problem studied in this paper.
The authors would like to thank Kaiyuan Gu, Ruicheng Hu, and Sheng Qin for assistance with the verification of an earlier blueprint of this paper.
The authors would like to thank Ruochuan Liu and Gang Tian for constant support and encouragement.
The first author would like to thank Florin Ambro, Federico Bongiorno, Paolo Cascini, Eric Riedl, Calum Spicer, Roberto Svaldi, and Zheng Xu for useful discussions.

\section{Notation, conventions, and preliminaries}\label{sec:notation}

We adopt the standard notation and terminology for the minimal model program from~\cite{Sho92,KM98,BCHM10} and use them freely.

\subsection{Conventions}

A \emph{variety} is an integral, separated, finite-type scheme over a field.
Throughout this paper, unless otherwise stated, \(X\) denotes a normal complex projective variety of dimension~\(n\), and \(T_X=\Omega_X^\vee\) denotes its tangent sheaf.

For a torsion-free coherent sheaf \(\Ee\) on \(X\) and a smooth complete curve \(C\subset X\) contained in the locally free locus of \(\Ee\), we write \(\deg(\Ee|_C)=c_1(\Ee|_C)\cdot[C]\), and we let \(\rk\Ee\) denote the generic rank of \(\Ee\).

Throughout the paper, we use the abbreviations \emph{HN} for \emph{Harder--Narasimhan}, \emph{MR} for \emph{Mehta--Ramanathan}, \emph{BM} for \emph{Bogomolov--McQuillan}, \emph{KST} for \emph{Kebekus--Sol\'a Conde--Toma}, and \emph{JLR} for \emph{Jovinelly--Lehmann--Riedl}.

A coherent subsheaf \(\Gg\subset\Ff\) is \emph{saturated} if the quotient \(\Ff/\Gg\) is torsion-free.
A family \(\{C_t\}_{t\in T}\) of curves in \(X\) parametrized by an irreducible scheme \(T\) is \emph{dominating} if the natural evaluation morphism from the total space to \(X\) is dominant; equivalently, the union \(\bigcup_{t\in T}C_t\) contains a dense open subset of \(X\).

\subsection{Foliations}

\begin{defn}\label{defn:foliation}
A \emph{foliation} \(\Ff\) on a normal variety \(X\) is a saturated coherent subsheaf \(\Ff\subset T_X\) of the tangent sheaf that is closed under the Lie bracket.
The \emph{rank} of \(\Ff\) is its generic rank \(r=\rk\Ff\), and the \emph{canonical class} of \(\Ff\) is the Weil divisor class
\[
K_{\Ff}\;=\;-c_1(\Ff)\;=\;-\det(\Ff),
\]
defined on the locally free locus and extended to all of \(X\) by reflexive extension.
\end{defn}

In particular, for a complete curve \(C\) lying in the smooth locus of \(X\) and in the locally free locus of \(\Ff\), one has
\[
\deg(\Ff|_C)\;=\;\det(\Ff)\cdot C\;=\;-K_{\Ff}\cdot C.
\]

\subsection{Mehta--Ramanathan restriction}\label{subsec:MR}

We use the MR restriction theorem in the form stated by KST~\cite[Notation~4]{KST07}: for a torsion-free coherent sheaf \(\Ee\) on a normal projective variety \(X\) and ample classes \(H_1,\ldots,H_{n-1}\) on \(X\), a general complete-intersection curve \(C=D_1\cap\cdots\cap D_{n-1}\) cut out by sufficiently large multiples \(D_i\in|m_iH_i|\) is smooth, lies in the locally free locus of \(\Ee\), and the HN filtration of \(\Ee\) restricts to the HN filtration of \(\Ee|_C\).
We call such a curve a \emph{sufficiently general MR curve} for \(\Ee\).

\section{Reduction to the algebraically integrable case}\label{sec:positive-HN}

We reduce Theorem~\ref{thm:main} to the algebraically integrable case using the following theorem.

\begin{thm}\label{thm:curve-level-positive-hn}
Let $X$ be a normal projective variety and $\Ff$ a foliation on $X$. Let $C$ be a sufficiently general complete intersection curve such that $C$ is contained in the smooth locus of $X$ and the locally free locus of $\Ff$. Assume that 
$$\deg(\Ff|_C)=-K_{\Ff}\cdot C>0.$$
Then there exists an algebraically integrable foliation $0\subsetneq\Gg\subset\Ff$ satisfying the following.
\begin{enumerate}
    \item $\Gg$ is regular along $C$.
    \item $\Gg|_C$ is ample.
    \item $-K_{\Gg}\cdot C\geq -K_{\Ff}\cdot C>0$. 
\end{enumerate}
\end{thm}
\begin{proof}
The proof follows essentially from \cite[Theorem 4.2]{CP19}. For the reader's convenience, we provide a full proof here. Since $C$ is contained in the smooth locus of $X$, possibly replacing $X$ with a resolution, we may assume that $X$ is smooth. Let $\alpha:=[C]$. Let
$$0=\Ff_0\subsetneq\Ff_1\subsetneq\Ff_2\dots\subsetneq\Ff_m=\Ff$$
be the HN filtration of $\Ff$ with respect to $\alpha$. Since $C$ is sufficiently general, we may assume that $C$ is a sufficiently general MR curve for $\Ff$. In particular, the HN filtration of $\Ff$ restricts to the HN filtration of $\Ff|_C$. We may further assume that $\Ff_i$ is regular along $C$ for any $1\leq i\leq m$. Denote by
$$\mathcal{Q}_i:=\Ff_i/\Ff_{i-1}\quad \text{and}\quad \mu_i:=\mu_{\alpha}(\mathcal{Q}_i)=\frac{c_1(\mathcal{Q}_i)\cdot C}{\rk\mathcal{Q}_i}.$$
Then
$$\mu_1>\mu_2>\cdots>\mu_m$$
by the definition of an HN filtration.
Since
$$\deg(\mathcal F|_C)=c_1(\mathcal F)\cdot C>0,$$
$\mu_i>0$ for some $i$, and we may let
$$s:=\max_{1\leq i\leq m}\{i\mid \mu_i>0\}.$$
We define $\Gg:=\Ff_s$. Then
$$\mu_{\alpha}^{\min}(\Gg)=\mu_s>0.$$
By \cite[Theorem 4.2]{CP19}, $\Gg$ is an algebraically integrable foliation. 

Since $C$ is a sufficiently general MR curve for $\Ff$, we have that
$$0=\Ff_0|_C\subset\Ff_1|_C\subset\dots\subset\Ff_m|_C=\Ff|_C$$
is an HN filtration of $\Ff|_C$. In particular,
\begin{equation}\label{equ:hn-gc}
    0=\Ff_0|_C\subset\Ff_1|_C\subset\dots\subset\Ff_s|_C=\Gg|_C
\end{equation}
is an HN filtration of $\Gg|_C$. Thus the slopes of the HN filtration \eqref{equ:hn-gc} are $\mu_1,\dots,\mu_s$, which are all positive. Thus
$$\mu_{\alpha}^{\min}(\Gg|_C)>0.$$
Since \(\Gg|_C\) is a vector bundle on the curve \(C\) and $\mu_{\alpha}^{\min}(\Gg|_C)>0$, all the non-zero quotient bundles of $\Gg|_C$ have positive degree, and so $\Gg|_C$ is an ample vector bundle.

By our construction of $\Gg$, all HN slopes of $(\Ff/\Gg)|_C$ are non-positive. Thus
\begin{equation}
    \deg\left((\Ff/\Gg)|_C\right)\leq 0
\end{equation}
so
$$-K_{\Gg}\cdot C=\deg(\Gg|_C)=\deg(\Ff|_C)-\deg\left((\Ff/\Gg)|_C\right)\geq\deg(\Ff|_C)=-K_{\Ff}\cdot C.$$
This implies (3). The theorem follows.
\end{proof}

\section{The graph-neighborhood relative model}\label{sec:graph-neighborhood}

\begin{defn}\label{defn:graph-neighborhood}
Let \(Y\) be a complex variety, \(\Ee\subset T_Y\) a foliation regular along a smooth complete subvariety \(\Sigma\subset Y\), and assume that \(\Sigma\) is transverse to \(\Ee\) and that \(\Ee|_\Sigma\) is ample.
The \emph{analytic graph-neighborhood} of \(\Sigma\) in \((Y,\Ee)\) is the unique germ of an analytic submanifold \(W^{\rm an}\subset Y\) containing \(\Sigma\), of dimension \(\dim\Sigma+\rk\Ee\), whose fibers over \(\Sigma\) along the transverse direction are analytic open subsets of \(\Ee\)-leaves.
The \emph{algebraic graph-neighborhood} of \(\Sigma\) is the Zariski closure of \(W^{\rm an}\) inside \(Y\); when the normal bundle \(\Ee|_\Sigma\) is ample, Hartshorne's theorem on formal functions implies that this Zariski closure has the same dimension as \(W^{\rm an}\) (see~\cite[Facts~16--19]{KST07}).
\end{defn}

In this section we use the data \((X,\Gg,C)\) produced by Theorem~\ref{thm:curve-level-positive-hn} to build a smooth projective relative model
\[
\rho\colon W\to B,\qquad \phi\colon W\to X,\qquad \sigma\colon B\to W,
\]
in which \(B\) is a smooth projective curve of positive genus with a finite morphism \(\nu\colon B\to C\), the morphism \(\rho\) is projective and flat of relative dimension \(s=\rk\Gg\) and is smooth near \(\sigma(B)\), and there is a canonical isomorphism
\[
T_{W/B}\big|_{\sigma(B)}\;\simeq\;(\phi\circ\sigma)^*\Gg.
\]
The construction is due to Bogomolov--McQuillan~\cite{BM16} and to KST~\cite{KST07}.
We split it into a sequence of lemmas, each making explicit the hypotheses used.

\begin{setup}\label{setup:graph-neighborhood-datum}
Let \(\Gg\subset T_X\) be a foliation of rank~\(s\), and let \(C\subset X\) be a complete curve contained in the smooth locus of~\(X\), with \(\Gg\) regular along~\(C\) and \(\Gg|_C\) ample.
Choose a finite morphism \(\nu\colon B\to C\) from a smooth projective curve of positive genus, and set \(e=\deg\nu\).
If \(C\) is additionally smooth, we call this a \emph{smooth} instance of the set-up.
\end{setup}

\begin{nota}\label{nota:construction-notation}
In the setting of Set-up~\ref{setup:graph-neighborhood-datum}, the BM/KST graph-neighborhood construction proceeds through three stages; we fix the following notation for the remainder of Sections~\ref{sec:graph-neighborhood}--\ref{sec:spreading}.
\begin{enumerate}[(I)]
\item \emph{Product trick.}
Apply the KST product trick~\cite[\S 4.1]{KST07} to \((X,C,\Gg)\): set \(Y=X\times B\) with projections \(p_1\colon Y\to X\) and \(p_2\colon Y\to B\), let \(\Gg_Y=p_1^*\Gg\subset T_{Y/B}\) be the pulled-back foliation, and let \(\sigma_0\colon B\to Y\) denote the graph of~\(\nu\).
\item \emph{Graph-neighborhood algebraization and source extraction.}
Apply the KST graph-neighborhood construction~\cite[\S 4, Facts~16--19]{KST07} to \(\sigma_0(B)\) inside \((Y,\Gg_Y)\): let \(\overline{W}\) denote the projective Zariski closure over \(B\) of the analytic graph-neighborhood; it has dimension \(s+1\) and contains \(\sigma_0(B)\) (Lemma~\ref{lem:kst-transverse-graph-algebraization}).
Let \(\overline{W}_0\subset\overline{W}\) be the irreducible component containing \(\sigma_0(B)\).
\item \emph{Desingularization.}
After normalization and projective desingularization of \(\overline{W}_0\), write \(W\) for the resulting smooth projective model, \(\rho\colon W\to B\) and \(\phi\colon W\to X\) for the structural morphisms, and \(\sigma\colon B\to W\) for the lifted section.
\end{enumerate}
The existence and properties of these objects are established in Lemmas~\ref{lem:kst-product-trick-transfer}--\ref{lem:graph-neighborhood-relative-model} below.
Some intermediate results (Lemmas~\ref{lem:kst-bm-graph-source-extraction} and~\ref{lem:flattened-graph-neighborhood-component}) work with the normalization of \(\overline{W}_0\) alone, without desingularization; in those statements, \(W\) denotes the normalization and may not be smooth.
\end{nota}

\subsection{The KST product trick}

\begin{lem}\label{lem:kst-product-trick-transfer}
Adopt the smooth instance of Set-up~\ref{setup:graph-neighborhood-datum}, with the notation of Notation~\ref{nota:construction-notation}(I).
Then \(\sigma_0(B)\) is smooth and transverse to \(\Gg_Y\), the restriction \(\Gg_Y|_{\sigma_0(B)}\simeq\nu^*(\Gg|_C)\) is ample, and the projection \(p_1\) maps every irreducible curve contained in a \(\Gg_Y\)-leaf to an irreducible curve tangent to \(\Gg\) on \(X\).
For every Cartier divisor \(H\) on \(X\) and every irreducible curve \(R\subset Y\), one has \((p_1^*H)\cdot R=H\cdot(p_1)_*R\) by the projection formula; in particular
\[
(p_1^*H)\cdot\sigma_0(B) \;=\; e\,H\cdot C,\qquad
\deg\Gg_Y|_{\sigma_0(B)} \;=\; e\,\deg\Gg|_C.
\]
\end{lem}

\begin{proof}
This is the KST product trick~\cite[\S 4.1]{KST07}: after replacing \(C\) by a finite smooth nonconstant cover of positive genus, replacing \(X\) by \(Y=X\times B\), taking the graph of \(\nu\), and pulling back \(\Gg\) from \(X\), one obtains a foliation \(\Gg_Y\subset T_{Y/B}\) for which the graph \(\sigma_0(B)\) is smooth and transverse to \(\Gg_Y\), the restriction \(\Gg_Y|_{\sigma_0(B)}\) is the finite pullback of \(\Gg|_C\), and leaves of \(\Gg_Y\) are isomorphic to leaves of \(\Gg\).

Finite pullback preserves ampleness of vector bundles on curves, so \(\Gg_Y|_{\sigma_0(B)}\) is ample.
Transversality follows because \(\Gg_Y\subset T_{Y/B}\), while \(T_{\sigma_0(B)}\) maps isomorphically to \(T_B\).
If a curve maps into a \(\Gg_Y\)-leaf inside a fiber of \(p_2\), then its image under \(p_1\) lies in the corresponding \(\Gg\)-leaf on \(X\), and is therefore tangent to \(\Gg\).
For a Cartier divisor \(H\) on \(X\), the projection formula gives \((p_1^*H)\cdot R=H\cdot(p_1)_*R\); in particular \((p_1^*H)\cdot\sigma_0(B)=H\cdot\nu_*[B]=e\,H\cdot C\).
The identity \(\deg\Gg_Y|_{\sigma_0(B)}=e\,\deg\Gg|_C\) follows in the same way.
\end{proof}

\subsection{Algebraization of the analytic graph-neighborhood}

\begin{lem}\label{lem:kst-transverse-graph-algebraization}
Let \(Y\) be a normal complex projective variety equipped with a morphism \(\pi\colon Y\to B\) to a smooth projective curve.
Let \(\Ee\subset T_{Y/B}\) be a rank \(s\) foliation, regular along a smooth section \(\sigma\colon B\to Y\), such that \(\sigma(B)\) is transverse to \(\Ee\) and \(\Ee|_{\sigma(B)}\) is ample.
Assume that, near the section, the leaves of \(\Ee\) are contained in the fibers of \(\pi\).
Then the KST graph-neighborhood construction of~\cite[\S 4, Facts~16--19]{KST07} produces an irreducible projective Zariski closure \(\overline{W}\) over \(B\), of dimension \(s+1\), containing the section.
Its normalization lifts the section, is smooth along the lifted section, and has relative tangent bundle along the section identified with \(\Ee|_{\sigma(B)}\).
Moreover, the fiber component containing the local analytic leaf is the algebraic leaf through the corresponding section point.
\end{lem}

\begin{proof}
By~\cite[Fact~16]{KST07}, applied in the transverse setup, there is an irreducible analytic submanifold \(W^{\rm an}\subset Y\) containing the section, smooth over \(B\), whose fibers are analytic open subsets of leaves and whose normal bundle along the section is \(\Ee|_{\sigma(B)}\).
By~\cite[Fact~17]{KST07}, Hartshorne's theorem on formal functions and the ampleness of this normal bundle imply that the Zariski closure has the same dimension as \(W^{\rm an}\), namely \(s+1\).
By~\cite[Corollary~18]{KST07}, the fiber component containing the analytic local leaf is the algebraic leaf.
By~\cite[Remark~19]{KST07}, normalization lifts the analytic graph-neighborhood, is smooth along the lifted section, and preserves the normal-bundle identification.
Since \(Y\) is projective over \(B\), the Zariski closure is projective over \(B\).
\end{proof}

\subsection{The KST/BM graph-neighborhood source extraction}

\begin{lem}\label{lem:kst-bm-graph-source-extraction}
Adopt the smooth instance of Set-up~\ref{setup:graph-neighborhood-datum}, with the notation of Notation~\ref{nota:construction-notation}(I)--(II).
After normalization \(W\to\overline{W}_0\), there are morphisms \(\rho\colon W\to B\) and \(\phi\colon W\to X\), and a dense open \(B^\circ\subset B\), with the following properties.
\begin{enumerate}[(i)]
\item \(\overline{W}_0\) is projective over \(B\), dominates \(B\), and has dimension \(s+1\).
\item The section lifts to \(\sigma\colon B\to W\), \(W\) is smooth near \(\sigma(B)\), and \(T_{W/B}|_{\sigma(B)}\simeq(\phi\circ\sigma)^*\Gg\).
\item For \(b\in B^\circ\), the unique fiber component through \(\sigma(b)\) maps into the \(\Gg\)-leaf through the corresponding point of \(C\).
\item Degrees against Cartier divisors pulled back from \(X\) are computed by the projection formula; section degrees are multiplied by \(\deg\nu\).
\end{enumerate}
\end{lem}

\begin{proof}
Lemma~\ref{lem:kst-product-trick-transfer} applies~\cite[\S 4.1]{KST07} to replace \((X,C,\Gg)\) by \((Y,\sigma(B),\Gg_Y)\), where the section is transverse to \(\Gg_Y\), \(\Gg_Y|_{\sigma(B)}\simeq\nu^*(\Gg|_C)\) is ample, and projection \(p_1\colon Y\to X\) sends \(\Gg_Y\)-leafwise curves to \(\Gg\)-tangent curves while preserving intersection numbers and scaling section degrees by \(\deg\nu\).

Lemma~\ref{lem:kst-transverse-graph-algebraization} then applies~\cite[Fact~16, Fact~17, Corollary~18, Remark~19]{KST07} to the transverse setup.
It gives the irreducible projective closure of dimension \(s+1\), the section lift after normalization, smoothness along the lifted section, the relative tangent identification, and the assertion that the fiber component through the analytic local leaf is the algebraic leaf.

Combining these two lemmas gives the asserted properties of the source curve.
The morphism \(\phi\colon W\to X\) is \(p_1\) composed with the graph-neighborhood map to \(Y\), and the projection formula gives the degree assertion.
\end{proof}

\begin{rem}
The package extracted in Lemma~\ref{lem:kst-bm-graph-source-extraction} is the same graph-neighborhood package as in~\cite[Fact~2.1.1, Better Fact~2.2.1]{BM16}: the graph-neighborhood algebraizes with dimension \(s+1\), carries a section over the base curve, has fibers mapping to invariant subvarieties through the corresponding points, and the pulled-back foliation maps to the relative tangent bundle near the section.
\end{rem}

\subsection{Flatness over a smooth curve}

The following lemma is standard; we state it for ease of later reference.

\begin{lem}\label{lem:integral-curve-base-flatness}
Let \(k\) be an algebraically closed field, let \(B\) be a smooth irreducible curve over \(k\), and let \(W\) be an integral \(k\)-variety of dimension \(s+1\).
If \(\rho\colon W\to B\) is a dominant morphism, then \(\rho\) is flat, and every irreducible component of every fiber of \(\rho\) has dimension \(s\).
\end{lem}

\begin{proof}
At a closed point \(b\in B\), the local ring \(\Oo_{B,b}\) is a discrete valuation ring (DVR), and \(\Oo_{W,w}\) is torsion-free over \(\Oo_{B,b}\) for every \(w\in\rho^{-1}(b)\) since \(\rho\) is dominant and \(W\) is integral.
Torsion-free modules over a DVR are flat~\cite[Tag~0539]{Stacks}, and hence \(\rho\) is flat at every point of \(W\).
Each closed fiber is cut out near the generic point of any of its components by a single nonzero non-unit (a uniformizer of \(\Oo_{B,b}\) pulled back to \(W\)), so by Krull's principal ideal theorem each component has codimension one in \(W\); since \(W\) is a variety over a field, this means dimension \(s\)~\cite[Tag~02IJ]{Stacks}.
\end{proof}

\subsection{Non-contraction along the section}

\begin{lem}\label{lem:section-branch-noncontraction}
Let \(W\) and \(X\) be complex varieties, let \(\phi\colon W\to X\) be a morphism, and let \(p\in W\) be a smooth point.
Let \(V\subset T_{W,p}\) be a linear subspace such that \(d\phi_p|_V\) is injective.
If \(R\subset W\) is an irreducible curve through \(p\) whose Zariski tangent space \(T_{R,p}\) contains a nonzero vector lying in \(V\), then \(\phi(R)\) is not a point.
\end{lem}

\begin{proof}
If \(\phi(R)\) were a point, then \(\phi|_R\) would be constant, so its differential would annihilate \(T_{R,p}\).
Choosing a nonzero vector \(v\in T_{R,p}\cap V\), we get \(d\phi_p(v)=0\), contradicting injectivity of \(d\phi_p|_V\).
Hence \(\phi(R)\) is nonconstant.
\end{proof}

\begin{lem}\label{lem:section-curve-pushforward}
Let \(B\) be a smooth complex curve, let \(W\) be a complex projective variety, and let \(\rho\colon W\to B\) have a section \(\sigma\colon B\to W\).
Let \(X\) be a normal complex variety, let \(\phi\colon W\to X\) be a morphism, and let \(\Gg\subset T_X\) be a rank \(s\) foliation regular along \(\phi(\sigma(B))\).
Assume that \(\rho\) is smooth near \(\sigma(B)\), that \(\phi\circ\sigma\) lies in the regular locus of \(\Gg\), that \(d\phi\) identifies \(T_{W/B}|_{\sigma(B)}\) with \((\phi\circ\sigma)^*\Gg\), and that for a general \(b\in B\), the fiber component through \(\sigma(b)\) maps into the \(\Gg\)-leaf through \(\phi(\sigma(b))\).
Then every irreducible curve \(R\) in that general fiber component and meeting \(\sigma(b)\) has nonconstant image tangent to \(\Gg\).
If \(R\) is rational, then \(\phi(R)\) is rational, and for every Cartier divisor \(H\) on \(X\),
\[
(\phi^*H)\cdot R\;=\;H\cdot\phi_*R.
\]
\end{lem}

\begin{proof}
Choose \(b\) general so that the leaf-containment hypothesis holds and \(\rho\) is smooth at \(\sigma(b)\).
Let \(R\subset W_b\) be an irreducible curve through \(\sigma(b)\) in that component.
Since \(\rho\) is smooth at \(\sigma(b)\), the fiber \(W_b\) is smooth there and
\[
T_{W_b,\sigma(b)}\;=\;T_{W/B,\sigma(b)}.
\]
The curve \(R\) has a nonzero Zariski tangent vector at \(\sigma(b)\), and this vector lies in \(T_{W/B,\sigma(b)}\).
Lemma~\ref{lem:section-branch-noncontraction}, applied with \(p=\sigma(b)\) and \(V=T_{W/B,\sigma(b)}\), shows that \(\phi(R)\) is not a point.

The component containing \(R\) maps into the \(\Gg\)-leaf through \(\phi(\sigma(b))\), so the nonconstant image is tangent to \(\Gg\).
If \(R\) is rational, the normalization of \(\phi(R)\) is dominated by the normalization of \(R\), hence \(\phi(R)\) is rational.
The intersection identity is the projection formula for \(\phi|_R\).
\end{proof}

\begin{lem}\label{lem:relative-model-consequences}
Let \(B\) be a smooth complex projective curve, let \(W\) be a smooth irreducible projective variety of dimension \(s+1\), and let \(\rho\colon W\to B\) be a projective dominant morphism with a section \(\sigma\colon B\to W\).
Let \(X\) be a normal complex variety, let \(\phi\colon W\to X\) be a morphism, and let \(\Gg\subset T_X\) be a rank \(s\) foliation regular along \(\phi(\sigma(B))\).
Assume that \(\rho\) is smooth near \(\sigma(B)\), that \(\phi\circ\sigma\) lies in the regular locus of \(\Gg\), that \(d\phi\) identifies \(T_{W/B}|_{\sigma(B)}\) with \((\phi\circ\sigma)^*\Gg\), and that for a general \(b\in B\), the fiber component through \(\sigma(b)\) maps into the \(\Gg\)-leaf through \(\phi(\sigma(b))\).
Then \(\rho\) is flat of relative dimension \(s\); every irreducible curve \(R\) in that general fiber component and meeting \(\sigma(b)\) has nonconstant image tangent to \(\Gg\); and if \(R\) is rational, then \(\phi(R)\) is rational and for every Cartier divisor \(H\) on \(X\),
\[
(\phi^*H)\cdot R\;=\;H\cdot\phi_*R.
\]
\end{lem}

\begin{proof}
Lemma~\ref{lem:integral-curve-base-flatness} applies because \(W\) is smooth and irreducible, hence integral, and because \(\rho\) is projective and dominant over the smooth curve \(B\); it gives flatness and fiber dimension \(s\).
Lemma~\ref{lem:section-curve-pushforward} applies to the remaining hypotheses and gives non-contraction, tangency, rationality of images of rational curves, and the projection formula.
\end{proof}

\subsection{The KST Assumption~20 relative model}

\begin{lem}\label{lem:kst-assumption20-relative-model}
Adopt the smooth instance of Set-up~\ref{setup:graph-neighborhood-datum}, with the notation of Notation~\ref{nota:construction-notation}(I)--(III).
Then \(\dim W=s+1\), \(\rho\) is flat of relative dimension \(s\), \(\rho\) is smooth near \(\sigma(B)\),
\[
\phi\circ\sigma\;=\;\nu,\qquad T_{W/B}|_{\sigma(B)}\;\simeq\;(\phi\circ\sigma)^*\Gg,
\]
the fiber component through \(\sigma(b)\) maps into the \(\Gg\)-leaf through \(\nu(b)\) for general \(b\in B\), and every rational curve in a general fiber through \(\sigma(b)\) pushes forward to a nonconstant rational curve tangent to \(\Gg\).
Intersections with Cartier divisors pulled back from \(X\) are computed by the projection formula, and
\[
(\phi^*H)\cdot\sigma(B)\;=\;e\,H\cdot C,\qquad \deg_\sigma T_{W/B}\;=\;e\,(-K_{\Gg}\cdot C).
\]
\end{lem}

\begin{proof}
We use~\cite[\S 4.1, Fact~16, Fact~17, Corollary~18, Remark~19, Assumption~20]{KST07}, and structure the construction in three steps.

\emph{Product trick.}
By~\cite[\S 4.1]{KST07}, after replacing \(C\) by the chosen finite cover \(\nu\colon B\to C\) and replacing \((X,C,\Gg)\) by \(Y=X\times B\), taking the graph of \(\nu\), and pulling back \(\Gg\) to \(\Gg_Y\subset T_{Y/B}\), the graph \(\sigma(B)\) is smooth and transverse to \(\Gg_Y\), the restriction satisfies \(\Gg_Y|_{\sigma(B)}\simeq\nu^*(\Gg|_C)\), and the leaves of \(\Gg_Y\) project isomorphically to the leaves of \(\Gg\).

\emph{Graph-neighborhood algebraization.}
In this transverse setup,~\cite[Fact~16]{KST07} produces an irreducible analytic graph-neighborhood \(W^{\rm an}\) containing the section, smooth over \(B\), whose fibers are analytic open subsets of leaves and whose normal bundle along the section is \(\Gg_Y|_{\sigma(B)}\).
By~\cite[Fact~17]{KST07}, Hartshorne's theorem on formal functions and ampleness of this normal bundle imply that the Zariski closure \(\overline{W}\) has the same dimension as \(W^{\rm an}\), namely \(s+1\).
By~\cite[Corollary~18]{KST07}, the fiber component containing the analytic local leaf is the algebraic leaf through the corresponding section point.
By~\cite[Remark~19]{KST07}, the universal property of normalization embeds \(W^{\rm an}\) in the normalization of \(\overline{W}\); in particular, the normalization is smooth along the lifted section, and the normal bundle there remains \(\Gg_Y|_{\sigma(B)}\).

\emph{Normalization and desingularization.}
Immediately after Remark~19,~\cite{KST07} replaces the space by a desingularization of the normalization of the graph-neighborhood closure, and~\cite[Assumption~20]{KST07} states the resulting model: a smooth total space, a smooth section of a morphism to the base curve, the morphism smooth near the section, and the relative tangent bundle along the section identified with the pulled-back foliation.
Applying this replacement to the component containing the section, and choosing a projective resolution that is an isomorphism over the open neighborhood of the lifted section where the normalization is already smooth, gives \(W\), \(\rho\), \(\sigma\), and a morphism \(W\to Y\).
Composing with \(p_1\colon Y=X\times B\to X\) gives \(\phi\colon W\to X\), and \(\phi\circ\sigma=\nu\).

It remains to verify the asserted properties of \(\rho\) and \(\phi\).
For general \(b\in B\), the component of \(W_b\) containing \(\sigma(b)\) maps to the algebraic \(\Gg_Y\)-leaf by~\cite[Corollary~18]{KST07}, hence to the corresponding \(\Gg\)-leaf by \(p_1\).
Since \(W\) is smooth and irreducible of dimension \(s+1\), Lemma~\ref{lem:integral-curve-base-flatness} gives flatness of \(\rho\) and dimension \(s\) for every fiber component.
Near the section,~\cite[Remark~19, Assumption~20]{KST07} identifies \(T_{W/B}|_{\sigma(B)}\) with \(\nu^*(\Gg|_C)\), and the differential of \(p_1\) identifies this with \((\phi\circ\sigma)^*\Gg\).
In particular \(d\phi\) is injective on \(T_{W/B,\sigma(b)}\); thus an irreducible curve \(R\subset W_b\) through \(\sigma(b)\) cannot be contracted by \(\phi\), and its image is rational and tangent to \(\Gg\) by leaf containment.
Finally, the projection formula yields the divisor intersections, and pullback by \(\nu\) gives the displayed degree equalities.
\end{proof}

\subsection{Normalization and the flattened component}

\begin{lem}\label{lem:curve-component-flatness-over-curve}
Let \(B\) be a smooth irreducible curve over an algebraically closed field.
Let \(\overline{W}_0\) be an irreducible projective variety of dimension \(s+1\) equipped with a dominant morphism \(\overline{\rho}_0\colon\overline{W}_0\to B\).
Let \(\nu\colon W\to\overline{W}_0\) be the normalization, and let \(\rho=\overline{\rho}_0\circ\nu\colon W\to B\).
Then \(\rho\) is projective and flat, and every irreducible component of every fiber has dimension \(s\).
If an open neighborhood of a section lifts to \(W\) and is smooth over \(B\), then \(\rho\) is smooth near the lifted section, and the relative tangent bundle along the section is the tangent bundle of that smooth relative neighborhood.
\end{lem}

\begin{proof}
The normalization \(W\to\overline{W}_0\) is finite, so \(W\) is integral, projective over \(B\), and dominates \(B\); flatness and the fiber-dimension assertion then follow from Lemma~\ref{lem:integral-curve-base-flatness}.
The final assertion is local on the lifted smooth neighborhood, on which the relative tangent bundle is by definition the kernel of \(T_W\to\rho^*T_B\).
\end{proof}

\begin{lem}\label{lem:flattened-graph-neighborhood-component}
With the notation of Notation~\ref{nota:construction-notation}(I)--(II), if \(W\) is the normalization of \(\overline{W}_0\), then \(W\to B\) is projective and flat of relative dimension \(s\), smooth near the lifted section, and every vertical rational curve through a general point of the section pushes forward to a curve tangent to \(\Gg\).
Degrees against divisors pulled back from the original variety are preserved by the projection formula.
\end{lem}

\begin{proof}
Lemma~\ref{lem:kst-bm-graph-source-extraction} gives the assertions about the source curve: the component containing the section is irreducible, dominant and projective over \(B\), of dimension \(s+1\), and admits the section lift, smoothness near the section, the relative tangent identification, and the leaf-containment property on general fibers.
Lemma~\ref{lem:curve-component-flatness-over-curve}, applied to this component, shows that its normalization is projective and flat over \(B\) with every fiber component of dimension \(s\), and preserves the smooth relative neighborhood of the section.

For a general \(b\in B^\circ\), the unique fiber component through the smooth point \(\sigma(b)\) maps into the \(\Gg\)-leaf through the corresponding point of \(C\).
Hence any rational curve \(R\subset W_b\) meeting \(\sigma(b)\) pushes forward to a rational curve tangent to \(\Gg\).
The projection formula gives \((\phi^*H)\cdot R=H\cdot\phi_*R\) for any Cartier divisor \(H\) pulled back from \(X\).
\end{proof}

\subsection{The packaged graph-neighborhood relative model}

\begin{lem}\label{lem:graph-neighborhood-relative-model}
Adopt Set-up~\ref{setup:graph-neighborhood-datum}, with the notation of Notation~\ref{nota:construction-notation}(I)--(III).
Then \(\phi\circ\sigma=\nu\), \(\rho\) is projective and flat of relative dimension \(s\), \(\rho\) is smooth near \(\sigma(B)\), \(T_{W/B}|_{\sigma(B)}\simeq(\phi\circ\sigma)^*\Gg\), the fiber component through \(\sigma(b)\) maps into the \(\Gg\)-leaf through \(\phi(\sigma(b))\) for general \(b\in B\), and every \(\rho\)-vertical rational curve meeting \(\sigma(b)\) for general \(b\) pushes forward to a rational curve tangent to \(\Gg\) with the same intersection against every Cartier divisor pulled back from \(X\).
Moreover \((\phi^*H)\cdot\sigma(B)=e\,H\cdot C\) and \(\deg_\sigma T_{W/B}=e\,(-K_{\Gg}\cdot C)\).
\end{lem}

\begin{proof}
The construction goes back to two parallel graph-neighborhood results.
By~\cite[Theorem~1]{KST07}, for a normal complex projective variety \(X\), a complete curve \(C\subset X_{\rm reg}\), and a foliation regular along \(C\) with ample restriction to \(C\), the leaves through points of \(C\) are algebraic, and the general such leaf is rationally connected.
The KST proof first passes to a finite smooth positive-genus cover of \(C\), then applies the product trick to make the lifted curve smooth and transverse, constructs an analytic graph-neighborhood, and finally algebraizes it via Hartshorne's theorem on formal functions, using ampleness of the normal bundle along the section.

The same package is given in graph-neighborhood language by~\cite[Fact~2.1.1, Better Fact~2.2.1]{BM16}.
Given a rank \(s\) weakly regular foliated variety and a smooth curve \(f\colon B\to X\) such that \(f^*T_{\Gg}\) is ample and the curve is neither contained in the singular locus nor generically tangent to the foliation, the Bogomolov--McQuillan construction produces a smooth algebraic variety \(W\) of dimension \(s+1\), morphisms \(\phi\colon W\to X\) and \(\rho\colon W\to B\), and a section \(\sigma\) such that the fibers of \(\rho\) map to invariant subvarieties through \(f(b)\), and the pulled-back foliation maps naturally to \(T_{W/B}\) near the section.
In the KST setup this map is an isomorphism along \(\sigma(B)\), since both sides identify with the normal bundle of the section in the graph-neighborhood.
The not-generically-tangent hypothesis of~\cite[Better Fact~2.2.1]{BM16} is automatic here: by Lemma~\ref{lem:kst-product-trick-transfer}, \(\sigma(B)\) is transverse to \(\Gg_Y\subset T_{Y/B}\) inside \(Y=X\times B\), and we have also arranged for it to lie in the smooth locus of \(Y\) and in the locally free locus of \(\Gg_Y\).

Apply these results to \((X,\Gg,C)\) after choosing the finite positive-genus cover \(\nu\colon B\to C\).
The required projective model is supplied by Lemma~\ref{lem:kst-assumption20-relative-model}, which packages the KST product trick, graph-neighborhood algebraization, normalization, and desingularization in the form needed here.
It gives the projective flat morphism \(\rho\colon W\to B\) of relative dimension \(s\), smoothness near \(\sigma(B)\), the relative tangent identification along the section, the assertion that a rational curve in a general fiber through \(\sigma(b)\) pushes forward to a nonconstant curve tangent to \(\Gg\), and the degree scaling by \(e\).
The equality of curve degrees is the projection formula
\[
(\phi^*H)\cdot R\;=\;H\cdot\phi_*R.\qedhere
\]
\end{proof}

\section{Stable maps and bend-and-shatter preliminaries}\label{sec:stable-maps}

This section assembles the technical lemmas about stable maps and Frobenius pullbacks needed in the next section's relative bend-and-shatter argument.

\begin{setup}\label{setup:relative-source}
Let \(k\) be an algebraically closed field, let \(B\) be a smooth projective curve over \(k\), and let \(\pi\colon W\to B\) be a projective morphism.
Let \(\Gamma\) be a smooth projective curve over \(k\), and fix \(h\colon \Gamma\to B\).
\end{setup}

\subsection{Verticality and tree extraction}

\begin{lem}\label{lem:vertical-stable-map-fiber-closure}
In Set-up~\ref{setup:relative-source}, assume further that \(\Gamma\) has genus \(g\).
Fix marked points \(a_1,\ldots,a_m\) on \(\Gamma\) so that the marked curve is stable.
Let \(S\) be an irreducible locally closed family of morphisms \(u\colon \Gamma\to W\) such that \(\pi\circ u=h\), viewed as stable maps with this fixed marked source, and let \(\beta := u_*[\Gamma]\in H_2(W,\mathbb Z)\) denote the (constant) image class of \(u\in S\) on \(W\).

Let \(\overline S^{\rm rel}\) be the closure of \(S\) inside the fiber of
\[
\overline M_{g,m}(W,\beta)\to\overline M_{g,m}(B,\pi_*\beta)
\]
over the stable map represented by \(h\colon \Gamma\to B\).
If \(v\colon \Gamma'\to W\) is a stable map in \(\overline S^{\rm rel}\), and the stabilization \(\tau\colon \Gamma'\to \Gamma\) contracts a rational tree \(T\subset \Gamma'\) to a point \(a\in \Gamma\), then every non-contracted component of \(v(T)\) is contracted by \(\pi\).
\end{lem}

\begin{proof}
Composition with \(\pi\), followed by stabilization, induces a morphism of stable-map stacks
\[
\overline M_{g,m}(W,\beta)\to\overline M_{g,m}(B,\pi_*\beta).
\]
For every map in \(S\), this composed stable map is the fixed point represented by \(h\colon \Gamma\to B\), and \(\overline S^{\rm rel}\) is by definition contained in the fiber over this point.
Hence every stable map in \(\overline S^{\rm rel}\) has the same stabilized composition to \(B\).

If a rational tree contracted by \(\tau\colon \Gamma'\to \Gamma\) mapped non-constantly to \(B\), it would survive as a non-contracted component in the stable map to \(B\), contradicting the assertion that the stabilized composition equals \(h\colon \Gamma\to B\).
Hence the tree is mapped to a fiber of \(\pi\), and its non-contracted image components are \(\pi\)-vertical rational curves.
\end{proof}

\begin{lem}\label{lem:marked-tree-root-component}
Let \(k\) be an algebraically closed field, let \(X\) be a projective variety over \(k\), and let \(u\colon \Gamma'\to X\) be a stable map.
Let \(T\subset \Gamma'\) be a connected tree of irreducible rational components.
Suppose \(q\in T\) is a marked smooth point with \(u(q)=p\), and suppose \(u\) does not contract all of \(T\) to the point \(p\).
Then \(T\) contains an irreducible rational component \(R\) such that the image \(u(R)\) is a nonconstant rational curve containing \(p\).
\end{lem}

\begin{proof}
Let \(T^\vee\) be the dual tree of \(T\), rooted at the unique irreducible component \(R_0\) containing \(q\).
If \(u|_{R_0}\) is nonconstant, then \(u(R_0)\) is a rational curve and contains \(u(q)=p\), so we are done.

Assume \(u|_{R_0}\) is constant.
Since \(u(q)=p\), the component \(R_0\) is contracted to \(p\).
Let \(T^\vee_p\subset T^\vee\) be the maximal connected subtree containing \(R_0\) whose components are all contracted to \(p\).
This subtree is proper, because \(u\) does not contract all of \(T\) to \(p\).
Since \(T^\vee\) is connected, there exists an edge from a vertex of \(T^\vee_p\) to a vertex \(R\notin T^\vee_p\).
Let \(z\) be the node joining the corresponding components.
The component in \(T^\vee_p\) maps \(z\) to \(p\), hence \(u(z)=p\) on \(R\) as well.
The restriction \(u|_R\) cannot be constant: if it were constant, then because \(u(z)=p\), it would be contracted to \(p\), contradicting maximality of \(T^\vee_p\).
Thus \(u(R)\) is a nonconstant image of a rational component and contains \(p\).
\end{proof}

\subsection{Frobenius pullback}

\begin{lem}\label{lem:frobenius-points-and-degrees}
Let \(k\) be an algebraically closed field of characteristic \(p>0\), let \(\Gamma\) be a smooth projective curve over \(k\), and let \(q=p^m\).
Let
\[
\Gamma_m\;=\;\Gamma^{(p^{-m})}\;:=\;\Gamma\times_{\Spec k,F_k^{-m}}\Spec k
\]
be the inverse Frobenius twist, and let \(\Phi_m\colon \Gamma_m\to \Gamma\) be the geometric Frobenius, the finite \(k\)-morphism of degree \(q\) induced locally by \(a\mapsto a^{p^m}\).
Let \(f\colon \Gamma\to X\) be a morphism to a projective variety and set \(f_m=f\circ\Phi_m\colon \Gamma_m\to X\).
Then \(\Phi_m\) is a universal homeomorphism and identifies the closed points of \(\Gamma_m\) with the closed points of \(\Gamma\).
In particular, if \(a_m\in \Gamma_m\) corresponds to \(a=\Phi_m(a_m)\in \Gamma\), then \(f_m(a_m)=f(a)\) as closed points of \(X\).
Moreover, for every Cartier divisor \(H\) on \(X\),
\[
H\cdot(f_m)_*\Gamma_m\;=\;q\,H\cdot f_*\Gamma,
\]
and for every vector bundle \(E\) on \(X\),
\[
\deg_{\Gamma_m}f_m^*E\;=\;q\,\deg_\Gamma f^*E.
\]
\end{lem}

\begin{proof}
Since \(k\) is algebraically closed, Frobenius on \(k\) is an automorphism.
The inverse twist \(\Gamma_m\) is again a smooth projective curve over \(k\), and the geometric Frobenius \(\Phi_m\colon \Gamma_m\to \Gamma\) is a finite \(k\)-morphism of degree \(q=p^m\).
It is a universal homeomorphism: on affine charts it is induced by the \(p^m\)-power map, so inverse images of prime ideals have the same radicals.
This gives the closed-point identification and the displayed equality of closed points in \(X\).

Finite pullback by a degree-\(q\) morphism of smooth projective curves multiplies degrees of line bundles by \(q\).
Applying this to \(f^*\Oo_X(H)\) gives the divisor-degree formula.
Applying the same line-bundle formula to \(f^*\det E\) gives \(\deg_{\Gamma_m}f_m^*E=q\,\deg_\Gamma f^*E\).
\end{proof}

\subsection{Closedness of the bounded vertical locus}

\begin{lem}\label{lem:bounded-vertical-incidence-closed}
Let \(k\) be an algebraically closed field, let \(\Gamma\) be a projective curve over \(k\), let \(W\) be a projective variety over \(k\), let \(\pi\colon W\to B\) be a projective morphism to a projective variety, and let \(\sigma\colon \Gamma\to W\) be a morphism.
Let \(H\) be an ample Cartier divisor on \(W\), and let \(N\ge 0\) be an integer.
Let \(Z_{\le N}\subset \Gamma\) be the set of points \(a\in \Gamma\) such that \(\sigma(a)\) lies on a rational curve \(R\subset W\) satisfying
\[
H\cdot R\;\le\;N,\qquad \pi(R)\text{ is a point.}
\]
Then \(Z_{\le N}\) is closed.
\end{lem}

\begin{proof}
If \(N=0\), then \(Z_{\le N}=\varnothing\), because an ample divisor has positive degree on every nonconstant curve.
Assume \(N\ge 1\).
Choose \(a>0\) such that \(A=aH\) is very ample.
For \(1\le e\le aN\), let \(\overline M_{0,1}(W,e)\) be the Kontsevich stack of one-pointed genus-zero stable maps to \(W\) of \(A\)-degree \(e\).
The finite union
\[
\overline M_{\le aN}(W)\;:=\;\coprod_{1\le e\le aN}\overline M_{0,1}(W,e)
\]
is a finite disjoint union of proper algebraic stacks.

The properness input is the algebraic Kontsevich stable-map properness theorem; we use the formulation in~\cite[Theorem~1]{FP97}: if \(Y\) is a projective scheme of finite type over a field, then the moduli stack of stable maps into \(Y\) with uniformly bounded genus and degree is a proper algebraic stack.
Applying this with \(Y=W\), genus \(0\), one marking, and \(A\)-degree bounded by \(aN\) gives the proper finite-type stack above; in particular it is universally closed over \(k\).

Choose an ample Cartier divisor \(L\) on \(B\).
In a family of stable maps to \(W\), the degree of the pullback of \(\pi^*L\) on the source fibers is locally constant; hence the locus where this degree is zero is closed.
Since \(L\) is ample, this is exactly the condition that every irreducible component of the composed stable map to \(B\) is contracted.
Let \(\overline M_{\le aN}^{\pi\text{-vert}}(W)\) denote this closed proper substack.
The fibered product
\[
I_{\le N}\;:=\;\overline M_{\le aN}^{\pi\text{-vert}}(W)\times_{W,\operatorname{ev},\sigma}\Gamma
\]
is proper over \(\Gamma\), so its image is closed because proper algebraic stacks are universally closed over schemes.

The image is exactly \(Z_{\le N}\).
A vertical rational curve \(R\) through \(\sigma(a_0)\) with \(H\cdot R\le N\) gives a point of \(I_{\le N}\) after normalizing \(R\) and marking a preimage of \(\sigma(a_0)\).
Conversely, a point of \(I_{\le N}\) gives a vertical genus-zero stable map of positive \(A\)-degree at most \(aN\) whose marked point maps to \(\sigma(a_0)\).
If the marked component is contracted, root the dual tree at it and take the first component adjacent to the maximal subtree contracted to \(\sigma(a_0)\); this component is nonconstant and its image contains \(\sigma(a_0)\).
Thus some vertical rational component \(R\) through \(\sigma(a_0)\) has \(A\cdot R\le aN\), hence \(H\cdot R\le N\).
Therefore \(Z_{\le N}\) is closed.
\end{proof}

\subsection{Integer degree discreteness and dimension counts}

\begin{lem}\label{lem:integer-degree-discreteness}
Let \(H\) be a Cartier divisor on a projective variety \(W\).
Let \(N\in\mathbb R\), and let \(N_m\to N\) be a sequence of real numbers.
If, for infinitely many \(m\), a curve \(R_m\subset W\) satisfies \(H\cdot R_m\le N_m\), then for all sufficiently large such \(m\),
\[
H\cdot R_m\;\le\;N.
\]
\end{lem}

\begin{proof}
For every integral curve \(R\subset W\), the intersection number \(H\cdot R\) is an integer.
For \(m\) large enough, \(N_m<\lfloor N\rfloor+1\).
If \(H\cdot R_m\le N_m\), then the integer \(H\cdot R_m\) is at most \(\lfloor N\rfloor\), hence is at most \(N\).
\end{proof}

\begin{lem}\label{lem:relative-fixed-point-dimension}
In Set-up~\ref{setup:relative-source}, assume further that \(\pi\) is flat with fibers of dimension \(s\).
Let \(M\) be an irreducible locally closed family of \(B\)-morphisms \(g\colon \Gamma\to W\).
If \(a_1,\ldots,a_k\in \Gamma\) are distinct and the fixed-point locus
\[
M(x_1,\ldots,x_k)\;=\;\{g\in M\mid g(a_i)=x_i\text{ for all }i\}
\]
is nonempty for points \(x_i\in W_{h(a_i)}\), then this fixed-point locus has an irreducible component of dimension at least \(\dim M-ks\).
In particular, if \(\dim M\ge k(s+1)\), it contains a \(k\)-dimensional locally closed subfamily.
\end{lem}

\begin{proof}
The evaluation morphism factors as
\[
M\to W_{h(a_1)}\times\cdots\times W_{h(a_k)}.
\]
Flatness of \(\pi\) gives \(\dim W_{h(a_i)}=s\), so the target has dimension \(ks\).
A nonempty fiber of a morphism from an irreducible finite-type family to a target of dimension \(ks\) has some irreducible component of dimension at least \(\dim M-ks\).
This standard dimension estimate proves the claim.
\end{proof}

\subsection{Density at a general point from many markings}

\begin{lem}\label{lem:general-point-density-from-many-markings}
Let \(\Gamma\) be an irreducible smooth projective curve over an algebraically closed field, let \(W\) be a projective variety, let \(\pi\colon W\to B\) be a projective morphism to a projective variety \(B\), let \(\sigma\colon \Gamma\to W\) be a morphism, and let \(H\) be an ample Cartier divisor on \(W\).
Let \(N\in\mathbb R\).
Suppose that for arbitrarily large integers \(k\), there exists a dense open subset \(U_k\subset \Gamma^k\) with the following property: for every tuple \((a_1,\ldots,a_k)\in U_k\), at least one point \(\sigma(a_i)\) lies on a rational curve \(R\subset W\) such that \(H\cdot R\le N\) and \(\pi(R)\) is a point.
Then a general point of \(\sigma(\Gamma)\) lies on a rational curve \(R\) such that \(H\cdot R\le N\) and \(\pi(R)\) is a point.
The same conclusion holds if the displayed bound is first obtained as \(H\cdot R\le N_m\) for a sequence \(N_m\to N\).
\end{lem}

\begin{proof}
Since \(H\) is Cartier and ample, \(H\cdot R\) is a positive integer on every nonconstant curve.
If \(N<1\), the dense-open hypothesis is impossible, so the statement is vacuous; hence assume \(N\ge 1\).
Then \(H\cdot R\le N\) is equivalent to \(H\cdot R\le N_0\), where \(N_0=\lfloor N\rfloor\).
By Lemma~\ref{lem:bounded-vertical-incidence-closed} applied with the integer bound \(N_0\), the locus \(Z_{\le N_0}\subset\Gamma\) of points whose image lies on a \(\pi\)-vertical rational curve of \(H\)-degree at most \(N_0\) is closed.
If \(Z_{\le N_0}\ne\Gamma\), then \((\Gamma\setminus Z_{\le N_0})^k\) is a nonempty open subset of \(\Gamma^k\) and must meet the dense open set \(U_k\), contradicting the defining property of \(U_k\).
Thus \(Z_{\le N_0}=\Gamma\).

For a sequence \(N_m\to N\), Lemma~\ref{lem:integer-degree-discreteness} discards finitely many terms and converts the \(N_m\)-bound into the \(N\)-bound; if the limiting \(N<1\), the large-index hypothesis is impossible by the same positivity argument, and otherwise the preceding closed-locus argument applies.
\end{proof}

\section{Relative bend-and-shatter and the characteristic-\texorpdfstring{\(p\)}{p} bound}\label{sec:char-p-BB}

\subsection{The relative JLR shattering lemma}

\begin{lem}\label{lem:relative-jlr-proposition-22}
In Set-up~\ref{setup:relative-source}, assume further that \(\pi\) is flat with fibers of dimension \(s\) and that \(\Gamma\) has genus \(g\).
Let \(M\subset\Mor_B(\Gamma,W;h)\) be an irreducible locally closed family of \(B\)-morphisms \(u\colon \Gamma\to W\), and let \(f\in M\).
Put
\[
k_0\;=\;\left\lfloor\frac{\dim M}{s+1}\right\rfloor.
\]
Assume \(k_0>0\) and \(2g-2+k_0>0\).
If \(H\) is an ample Cartier divisor and all maps in \(M\) have \(H\)-degree \(e=H\cdot f_*\Gamma\), then for a general choice of distinct points \(a_1,\ldots,a_{k_0}\in \Gamma\), the stable-map closure of \(M\) over the fixed stable map \(h\colon \Gamma\to B\) contains \(k_0\) nonconstant \(\pi\)-vertical rational curves \(R_i\), with \(R_i\) passing through \(f(a_i)\), and
\[
\sum_iH\cdot R_i\;\le\;e.
\]
In particular one \(R_i\) has \(H\cdot R_i\le e/k_0\).
\end{lem}

\begin{proof}
For general distinct \(a_i\), the marked curve \((\Gamma,a_1,\ldots,a_{k_0})\) is stable because \(2g-2+k_0>0\).
Evaluation at \(a_i\) lands in the fiber \(W_{h(a_i)}\), of dimension \(s\).
Hence fixing \(u(a_i)=f(a_i)\) imposes codimension at most \(s\), and the common fixed locus has a locally closed component of dimension at least \(\dim M-k_0s\ge k_0\).
Choose a \(k_0\)-dimensional subfamily \(S\) of this fixed locus.

Map \(S\) to the Kontsevich stable-map stack with the fixed marked source and take its closure inside the fiber over the fixed stable map \(h\colon \Gamma\to B\).
By JLR~\cite[Lemma~2.1]{JLR25b}, a \(k\)-dimensional locally closed family of stable maps from a fixed stable marked curve, fixing the first \(k\) marked images, has a stable-map limit whose stabilization is the original marked curve and which contains a non-contracted rational tree at each fixed marking.
Applying this to \(S\) gives rational trees \(T_i\) over the \(a_i\), not contracted by the stable map and containing the fixed points \(f(a_i)\).

Lemma~\ref{lem:vertical-stable-map-fiber-closure} makes every non-contracted component of each \(T_i\) \(\pi\)-vertical, since the closure was taken over the fixed stable map to \(B\).
Lemma~\ref{lem:marked-tree-root-component} extracts from each \(T_i\) an irreducible nonconstant rational component whose image contains \(f(a_i)\).
Call these curves \(R_i\).

The \(T_i\) are exceptional trees over distinct markings, so the chosen components are distinct.
The stable limit has the same \(H\)-degree \(e\) as the maps in \(M\), and all effective nonconstant components have nonnegative \(H\)-degree.
Thus \(\sum_iH\cdot R_i\le e\), and the pigeonhole bound follows.
\end{proof}

\subsection{The characteristic-\texorpdfstring{\(p\)}{p} BM/JLR bound}

\begin{lem}\label{lem:bm-jlr-relative-positive-characteristic}
In Set-up~\ref{setup:relative-source}, assume further that \(k\) has characteristic \(p>0\), that \(\pi\) is flat of relative dimension \(s\) and smooth in a neighborhood of the image of \(f\).
In the relative setup of~\cite[Proposition~3.1.1]{BM16}, let \(f\colon \Gamma\to W\) be a \(B\)-morphism with finite composition \(h:=\pi\circ f\colon \Gamma\to B\).
Let \(H\) be an ample Cartier divisor on \(W\), and assume
\[
\delta\;=\;\deg_\Gamma f^*T_{W/B}\;>\;0.
\]
Then through a general point of \(f(\Gamma)\) there exists a rational curve \(R\subset W\), contracted by \(\pi\), such that
\[
H\cdot R\;\le\;(s+1)\,\frac{H\cdot f_*\Gamma}{\delta}.
\]
\end{lem}

\begin{proof}
Let \(\Gamma_m=\Gamma^{(p^{-m})}\) be the inverse Frobenius twist and let \(\Phi_m\colon \Gamma_m\to \Gamma\) be the geometric Frobenius, a finite \(k\)-morphism of degree \(q_m=p^m\).
Set \(f_m=f\circ\Phi_m\colon \Gamma_m\to W\) and \(h_m=\pi\circ f_m\colon \Gamma_m\to B\).
Lemma~\ref{lem:frobenius-points-and-degrees} identifies closed points of \(\Gamma_m\) and \(\Gamma\), gives \(f_m(a_m)=f(\Phi_m(a_m))\), and scales degrees by \(q_m\); in particular
\[
\deg_{\Gamma_m}f_m^*T_{W/B}\;=\;q_m\delta,\qquad H\cdot(f_m)_*\Gamma_m\;=\;q_m(H\cdot f_*\Gamma).
\]
Let \(M_m\subset\Hom_B(\Gamma_m,W;h_m)\) be an irreducible component through \(f_m\).
The relative tangent-obstruction calculation of~\cite[Proposition~3.1.1]{BM16} replaces \(T_W\) by \(T_{W/B}\): the tangent space to \(\Hom_B(\Gamma_m,W;h_m)\) at \(f_m\) is \(H^0(\Gamma_m,f_m^*T_{W/B})\), the obstruction lies in \(H^1(\Gamma_m,f_m^*T_{W/B})\), and every component through \(f_m\) has dimension at least \(h^0-h^1\).
Hence
\[
\dim M_m\;\ge\;\chi(\Gamma_m,f_m^*T_{W/B})\;=\;q_m\delta+s(1-g(\Gamma_m))\;=\;q_m\delta+s(1-g(\Gamma)).
\]

Put \(k_m=\lfloor\dim M_m/(s+1)\rfloor\).
For \(m\gg 0\), \(k_m>0\) and \(2g(\Gamma_m)-2+k_m>0\).
On the irreducible Hom component \(M_m\), the \(H\)-degree is constant: the universal map over \(\Gamma_m\times M_m\) pulls back \(\Oo_W(H)\) to a line bundle whose degree on the \(\Gamma_m\)-fibers is locally constant.
At \(f_m\) this degree is \(H\cdot(f_m)_*\Gamma_m=q_m(H\cdot f_*\Gamma)\).
Apply Lemma~\ref{lem:relative-jlr-proposition-22} to \(M_m\).
Its input is exactly the fixed-base relative Hom family: fixing points costs at most the fiber dimension \(s\), bend-and-shatter produces rational trees at the fixed markings, and closure over the fixed stable map to \(B\) makes their non-contracted components \(\pi\)-vertical.
Thus for a general \(k_m\)-tuple \(a_{m,1},\ldots,a_{m,k_m}\in \Gamma_m\), the stable limit contains \(k_m\) vertical nonconstant rational curves through the fixed marked images \(f_m(a_{m,i})\), and one of them satisfies
\[
H\cdot R_m\;\le\;\frac{q_m(H\cdot f_*\Gamma)}{k_m}.
\]
The dimension lower bound gives
\[
k_m\;\ge\;\frac{q_m\delta+s(1-g(\Gamma))}{s+1}-1.
\]
Consequently, for every \(\eta>0\) and all sufficiently large \(m\),
\[
\frac{q_m}{k_m}\;\le\;\frac{s+1}{\delta}+\eta.
\]
Thus the right side in the preceding degree estimate is at most
\[
\left(\frac{s+1}{\delta}+\eta\right)(H\cdot f_*\Gamma)
\]
for all large \(m\).
Since $H\cdot R_m$ are integers and \(\eta>0\) is arbitrary, Lemma~\ref{lem:integer-degree-discreteness} gives, for sufficiently large \(m\), a vertical rational component satisfying the stated numerical bound.

Lemma~\ref{lem:general-point-density-from-many-markings} then turns the construction from \(k_m\) general markings into a curve through a general point of \(f(\Gamma)\): the universal homeomorphism \(\Phi_m\colon \Gamma_m\to \Gamma\) sends general tuples on \(\Gamma_m\) to general tuples on \(\Gamma\), and Lemma~\ref{lem:frobenius-points-and-degrees} gives \(f_m(a_{m,i})=f(\Phi_m(a_{m,i}))\).
If the bounded-curve locus did not dominate \(\Gamma\), then a general \(k_m\)-tuple of \(\Gamma\) could avoid it, contradicting the preceding construction transported from \(\Gamma_m\).
\end{proof}

\section{Characteristic-zero spreading and the optimal bend-and-break bound}\label{sec:spreading}

\subsection{Characteristic-zero vertical spreading}

\begin{lem}\label{lem:characteristic-zero-vertical-spreading}
Let \(S\) be an integral finite-type \(\mathbb Z\)-scheme with characteristic-zero generic point.
Let \(\mathcal X\), \(\mathcal B\), and \(\mathcal C\) be projective \(S\)-schemes, where \(\mathcal B\to S\) and \(\mathcal C\to S\) are families of smooth projective curves, and let
\[
\pi\colon\mathcal X\to\mathcal B,\qquad \sigma\colon\mathcal C\to\mathcal X
\]
be \(S\)-morphisms with \(\pi\circ\sigma\colon\mathcal C\to\mathcal B\) prescribed.
Let \(\mathcal H\) be a relatively ample Cartier divisor on \(\mathcal X/S\), and let \(N\ge 1\) be an integer.
If for a Zariski-dense set of closed points \(s\in S\) a general point of \(\sigma_s(\mathcal C_s)\) lies on a rational curve \(R_s\subset\mathcal X_s\) contracted by \(\pi_s\) and satisfying \(\mathcal H_s\cdot R_s\le N\), then the same statement holds over the characteristic-zero generic fiber, and hence after base change along any embedding of the function field of \(S\) into \(\mathbb C\).
\end{lem}

\begin{proof}
For \(1\le d\le N\), let \(\overline M_{0,1}(\mathcal X/S,d)\) be the Kontsevich stack of one-pointed genus-zero stable maps to \(\mathcal X/S\) of \(\mathcal H\)-degree \(d\).
Since \(\mathcal X\to S\) is projective and \(N\) is fixed, the finite union over \(d\le N\) is proper over \(S\).
Let \(\mathcal V_{\le N}\subset\coprod_{d\le N}\overline M_{0,1}(\mathcal X/S,d)\) be the closed substack of stable maps that are vertical for \(\pi\), i.e.\ whose composition with \(\pi\) has degree zero on every component; this is closed because it is the inverse image of the degree-zero locus in the proper stable-map stack of \(\mathcal B/S\).

The fibered product
\[
I_{\le N}\;:=\;\mathcal V_{\le N}\times_{\mathcal X}\mathcal C
\]
is proper over \(\mathcal C\), so its image \(Z\subset\mathcal C\) is closed.
Fiberwise, \(Z_s\) is exactly the set of points of \(\mathcal C_s\) whose image in \(\mathcal X_s\) lies on a \(\pi_s\)-vertical rational curve of \(\mathcal H_s\)-degree at most \(N\).
If \(Z_\eta\) were a proper closed subset of the smooth curve \(\mathcal C_\eta\), then, after shrinking \(S\), its closure in \(\mathcal C\) would be finite over \(S\) and hence proper in every nearby closed fiber, contradicting the assumed dense set of closed fibers in which \(Z_s\) contains a dense open subset of \(\mathcal C_s\).
Hence \(Z_\eta\) is dense in \(\mathcal C_\eta\), and being closed it equals \(\mathcal C_\eta\).
\end{proof}

\begin{rem}
Lemma~\ref{lem:characteristic-zero-vertical-spreading} is the closed-locus form of the characteristic-zero spreading step in the proof of~\cite[Theorem~1.1]{JLR25b}: the extension from characteristic \(p\) to characteristic zero uses spreading-out as in~\cite[Theorem~5]{MM86}, and nef real divisors are handled as in~\cite[Theorem~II.5.8]{Kol96}.
The additional verticality condition here is closed in the same proper substack.
\end{rem}

\subsection{Optimal bend-and-break for ample subfoliations}

\begin{lem}\label{lem:ample-subfoliation-optimal-length}
Let \(\Gg\subset T_X\) be a rank \(s\) foliation, and let \(C\subset X\) be a sufficiently general complete-intersection curve such that \(\Gg\) is regular along \(C\) and \(\Gg|_C\) is ample.
Let \(H\) be an ample divisor.
Then through a general point of \(C\) there exists a rational curve \(\Sigma\) tangent to \(\Gg\) satisfying
\[
H\cdot\Sigma\;\le\;(s+1)\,\frac{H\cdot C}{-K_{\Gg}\cdot C}.
\]
\end{lem}

\begin{proof}
Since \(\Gg|_C\) is ample, \(-K_{\Gg}\cdot C=\deg(\Gg|_C)>0\).
Choose a finite cover \(\nu\colon B\to C\) of degree \(e\) from a smooth projective curve of positive genus, and let
\[
\rho\colon W\to B,\qquad \phi\colon W\to X,\qquad \sigma\colon B\to W
\]
be the BM/KST graph-neighborhood model produced by Lemma~\ref{lem:graph-neighborhood-relative-model}, smooth along the section, with \(T_{W/B}|_{\sigma(B)}\simeq(\phi\circ\sigma)^*\Gg\).
The relevant degrees are
\[
(\phi^*H)\cdot\sigma(B)\;=\;e\,H\cdot C,\qquad \deg_\sigma T_{W/B}\;=\;e\,(-K_{\Gg}\cdot C).
\]

Choose an ample Cartier divisor \(A\) on \(W\).
For every positive rational \(\varepsilon\), the \(\mathbb Q\)-Cartier divisor
\[
L_\varepsilon\;=\;\phi^*H+\varepsilon A
\]
is ample on \(W\).
We spread the graph-neighborhood data and \(L_\varepsilon\) over a finitely generated \(\mathbb Z\)-subalgebra \(S\) of \(\mathbb C\) in the standard way (cf.~\cite[Theorem~II.5.10]{Kol96}): any such finite-type datum spreads, and the properties constructible on \(\Spec S\), namely smoothness of \(\rho\) near the section, ampleness and Cartier class of \(L_\varepsilon\), the isomorphism class of \(T_{W/B}|_{\sigma(B)}\), and the values of the intersection numbers \(L_\varepsilon\cdot\sigma(B)\) and \(e\,(-K_{\Gg}\cdot C)\), all hold on a Zariski-dense open subset of \(\Spec S\), and hence at closed fibers of arbitrarily large residue characteristic.
Applying Lemma~\ref{lem:bm-jlr-relative-positive-characteristic} (with the abstract source curve \(\Gamma\) instantiated to the cover \(B\), and with \(f=\sigma\), \(h=\mathrm{id}_B\)) to a Cartier multiple of \(L_\varepsilon\) produces, through a general point of the section in such a closed fiber, a vertical rational curve of \(L_\varepsilon\)-degree at most
\[
(s+1)\,\frac{L_\varepsilon\cdot\sigma(B)}{e\,(-K_{\Gg}\cdot C)}.
\]
For vertical rational curves through a general point of the section, Lemma~\ref{lem:graph-neighborhood-relative-model} identifies their pushforwards with \(\Gg\)-tangent rational curves on \(X\).

We now pass to characteristic zero by Lemma~\ref{lem:characteristic-zero-vertical-spreading} (with \(\mathcal X = W\), \(\mathcal B = \mathcal C = B\), \(\pi = \rho\), and \(\sigma\) the section, all spread to families over the integral base \(\Spec S\), and \(\mathcal H\) a Cartier multiple of \(L_\varepsilon\)), applied to the corresponding integer degree bound; the hypotheses are met because the graph-neighborhood data, the section, the divisor \(L_\varepsilon\), and the verticality condition all spread over a finitely generated \(\mathbb Z\)-algebra.
Therefore, in characteristic zero, a general point of the section lies on a vertical rational curve \(R_\varepsilon\subset W\) satisfying
\[
L_\varepsilon\cdot R_\varepsilon\;\le\;(s+1)\,\frac{L_\varepsilon\cdot\sigma(B)}{e\,(-K_{\Gg}\cdot C)}.
\]
Since \(\phi^*H\cdot R_\varepsilon\le L_\varepsilon\cdot R_\varepsilon\) and \(L_\varepsilon\cdot\sigma(B)=e\,H\cdot C+\varepsilon\,A\cdot\sigma(B)\), the \(H\)-degree of the pushforward satisfies
\[
H\cdot\phi_*R_\varepsilon\;\le\;(s+1)\,\frac{e\,H\cdot C+\varepsilon\,A\cdot\sigma(B)}{e\,(-K_{\Gg}\cdot C)}.
\]
Set
\[
b_0\;=\;(s+1)\,\frac{H\cdot C}{-K_{\Gg}\cdot C},\qquad
b_\varepsilon\;=\;(s+1)\,\frac{e\,H\cdot C+\varepsilon\,A\cdot\sigma(B)}{e\,(-K_{\Gg}\cdot C)}.
\]
Then \(b_\varepsilon\to b_0\) as \(\varepsilon\to 0^+\).
For arbitrarily small \(\varepsilon>0\), the preceding construction yields a nonconstant rational curve through a general point of the section whose \(H\)-degree is a positive integer at most \(b_\varepsilon\); in particular \(b_0\ge 1\).
Choose \(\varepsilon\) so small that \(b_\varepsilon<\lfloor b_0\rfloor+1\).
For the dense open subset of section points supplied by this single \(\varepsilon\), every pushforward curve has integral \(H\)-degree at most \(b_\varepsilon\), hence at most \(\lfloor b_0\rfloor\le b_0\).
We conclude that a general point of the section lies on a vertical rational curve \(R\subset W\) with
\[
H\cdot\phi_*R\;\le\;(s+1)\,\frac{H\cdot C}{-K_{\Gg}\cdot C}.
\]
By Lemma~\ref{lem:graph-neighborhood-relative-model}, \(\Sigma:=\phi(R)\) is a rational curve tangent to \(\Gg\) through the corresponding general point of \(C\), and the projection formula gives \(H\cdot\Sigma\le(\phi^*H)\cdot R\).
This proves the lemma.
\end{proof}

\section{Proof of the main theorem}\label{sec:proof}

\begin{proof}[Proof of Theorem~\ref{thm:main}]
For \(m_i\gg 0\) and general \(D_i\), the curve \(C\) avoids the codimension-two singular and non-locally-free loci relevant to \(X\) and \(\Ff\).
Since \(K_{\Ff}\cdot C<0\),
\[
\deg(\Ff|_C)\;=\;\det(\Ff)\cdot C\;=\;-K_{\Ff}\cdot C\;>\;0.
\]
Apply Theorem~\ref{thm:curve-level-positive-hn} to obtain a positive-rank subfoliation \(\Gg\subset\Ff\), of rank \(s\le r\), such that \(\Gg|_C\) is ample and
\[
-K_{\Gg}\cdot C\;\ge\;-K_{\Ff}\cdot C.
\]
Lemma~\ref{lem:ample-subfoliation-optimal-length} gives a rational curve \(\Sigma\) through a general point of \(C\), tangent to \(\Gg\), with
\[
H\cdot\Sigma\;\le\;(s+1)\,\frac{H\cdot C}{-K_{\Gg}\cdot C}.
\]
Since \(\Gg\subset\Ff\), the curve is tangent to \(\Ff\).
Since \(s\le r\) and \(-K_{\Gg}\cdot C\ge -K_{\Ff}\cdot C>0\),
\[
(s+1)\,\frac{H\cdot C}{-K_{\Gg}\cdot C}\;\le\;(r+1)\,\frac{H\cdot C}{-K_{\Ff}\cdot C}.
\]
This proves the required bound.
\end{proof}

\begingroup
\sloppy
\emergencystretch=3em

\endgroup

\end{document}